\documentclass{article}
\usepackage{graphicx} 

\usepackage[backend=biber,
            sorting=nyt,
            giveninits=true,
            maxnames=99]{biblatex}

\DeclareFieldFormat[article]{title}{#1}
\DeclareFieldFormat[article]{journaltitle}{\mkbibemph{#1}}
\DeclareFieldFormat{volume}{#1}
\DeclareFieldFormat{number}{(#1)}
\DeclareFieldFormat{pages}{pp.~#1}
\DeclareFieldFormat{year}{#1}

\renewbibmacro*{journal+issuetitle}{%
  \usebibmacro{journal}%
  \setunit{\space}%
  \printfield{volume}%
  \iffieldundef{number}
    {}
    {\printfield{number}}%
  \setunit{\addcomma\space}%
  \printfield{year}%
  \newunit}

\renewbibmacro*{publisher+location+date}{%
  \printlist{location}%
  \iflistundef{publisher}
    {}
    {\setunit*{\addcolon\space}%
     \printlist{publisher}}%
  \setunit*{\addcomma\space}%
  \printfield{year}%
  \newunit}

\usepackage{tikz-cd}

\usepackage{amsmath,amsfonts,amssymb,mathtools}

\usepackage{graphicx,float}
\usepackage{amsthm}

\newtheorem{remark}{Remark}

\theoremstyle{definition}
\newtheorem{Rem}{Remark}[section]

\theoremstyle{definition}

\theoremstyle{definition}
\newtheorem{cor}{Corollary}[section]

\theoremstyle{definition}

\theoremstyle{definition}
\newtheorem{definition}{Definition}[section]

\theoremstyle{definition}
\newtheorem{thm}{Theorem}[section]

\theoremstyle{definition}
\newtheorem{lemma}{Lemma}[section]

\theoremstyle{definition}
\newtheorem{prop}{Proposition}[section]

\theoremstyle{definition}

\theoremstyle{definition}

\theoremstyle{definition}

\newcommand{\epn}{\operatorname{epn}}

\usepackage[ruled,vlined]{algorithm2e}
\usepackage{algorithmic}

\title{On the Secure Domination of Mycielskian Graphs}

\author{Kiran Bhutani, Anthony Christiana, and Peter Ulrickson}

\date{March 2026}

\begin{document}

\maketitle

\begin{abstract}

  \noindent 

  We study the secure domination number of the Mycielskian graph of a simple, connected, undirected graph. We give generally applicable bounds, compute secure domination numbers for Mycielskians of important families of graphs, and construct families of graphs realizing particular values of (secure) domination parameters.
\end{abstract}

\begin{center}
\end{center}

\section{Introduction}

This article investigates two graph-theoretic notions jointly: the secure domination number and the Mycielskian construction. We begin by briefly recalling each notion in turn.

\subsection{Secure Domination}
The domination number of a graph $G$, denoted $\gamma(G)$, is the minimum cardinality of a dominating set of $G$. Thus, $\gamma(G) =1$ if and only if $G$ has a vertex of degree $n-1$ where $n$ is the order of $G$. We call a vertex of degree $n-1$ a \textit{dominating vertex} of $G$. For more on graph domination see, for instance \cite{h1}.

\

A dominating set $S$ of a graph $G$ is called a secure dominating set of $G$ \cite{1} if for any $u \in V(G)-S, \exists  v\in S$ such that $S-\{v\} \cup \{u\}$ is a dominating set of $G$. The cardinality of a minimum secure dominating set of G is called the secure domination number of G and is denoted by $\gamma_s(G)$. 

\

If two vertices in a graph are not connected by an edge, a guard located at one vertex does not defend the other vertex. This observation yields the following result about complete graphs.

\begin{thm}\cite{ck}
  $\gamma_s(G) =1$ if and only if $G$ is a complete graph.
\end{thm}

\

\begin{remark}
    Let $P_n$ denote the path on $n$ vertices. A secure dominating set is obtained by choosing every other vertex (the second, fourth, and sixth) in each block of seven adjacent vertices. More formally, we have this theorem, which also involves cycles $C_n$.
\end{remark}

\begin{thm}\label{secDomPath}\cite{ck}
  $\gamma_s(P_n)=\gamma_s(C_n) =\lceil \frac{3n}{7}\rceil $
\end{thm}

\subsection{The Mycielskian of a Graph}

The Mycielskian of a graph $G$, denoted $\mu(G)$, is a supergraph of $G$ which was first introduced by Jan Mycielski in 1955 \cite{Myc}. Mycielski noted that the graph $\mu^i(K_2)$ is triangle-free, $(i-1)$-vertex-connected, and has chromatic number $(i+1)$. The Mycielskian is sometimes called the \textit{cone over} $G$.
\

The Mycielski construction increases the domination number by 1, which is to say ${\gamma(\mu(G)) = \gamma(G) + 1}$ \cite{FM}. However, outside of some work on ``generalized Mycielskians,'' little is known about the secure domination number of Mycielskian graphs.  \cite{NaRa}

This work is one part of a larger program: to relate the secure domination of a given graph to the secure domination number of a graph resulting from some $n$-ary operation. Here we consider the Mycielski construction as a unary operation on graphs.

\

\begin{definition}
  The Mycielskian graph (see \cite{Myc}) of a graph $G$ of order $n$ and size $m$ is a supergraph of $G$, denoted $\mu(G)$, with order $2n + 1$ and size $3m+n$. The Mycielski of a graph $G$ is constructed as follows:
\begin{enumerate}
    \item Let the vertices of $G$ be $\{v_1, v_2, \dots , v_{n}\}$.
    \item Take a graph $K_{1,n}$ whose set of vertices is disjoint from those of $G$. Denote by $w$ the central vertex of $K_{1,n}$ and label the $n$ other vertices of $K_{1,n}$ as $U = \{u_1, u_2, \dots, u_{n}\}$. The set $\{v_1, v_2, \dots , v_{n},u_1, u_2, \dots, u_{n},w\}$ is the set of vertices of $\mu(G)$.
    \item The edges of $\mu(G)$ are the edges of $G$ and $K_{1,n}$ together with the following:
 for each edge $v_i v_j \in E(G)$, the edges $u_i v_j$ and $v_i u_j$ are in $E(\mu(G))$.

\end{enumerate}
\end{definition}

\

We will consistently use the symbols $w$ and $u_i$ in the manner of the preceding definition. In other words, the vertex of degree $n$ (in the subgraph $K_{1,n}$)  will be referred to as $w$, and the leaf of $K_{1,n}$ associated to the vertex $v_i \in G$ will be referred to as $u_i$. We will call $u_i$ the \textit{image} of $v_i$. See Figure~\ref{fig:muP4}.

\begin{figure}
    \centering

    \begin{center}
\begin{tikzpicture}[scale=1.2, colorstyle/.style={circle,draw=black!100,thick,inner sep=2pt, minimum size=3mm}]
 \node (a1) at (0,0) [colorstyle,label=below:$v_1$]{};
 \node (a2) at (1,0) [colorstyle,label=below:$v_2$]{};
 \node (a3) at (2,0) [colorstyle,label=below:$v_3$]{};
 \node (a4) at (3,0) [colorstyle,label=below:$v_4$]{};
 \node (b1) at (0,1) [colorstyle,label=left:$u_1$]{};
 \node (b2) at (1,1) [colorstyle,label=left:$u_2$]{};
 \node (b3) at (2,1) [colorstyle,label=left:$u_3$]{};
 \node (b4) at (3,1) [colorstyle,label=left:$u_4$]{};
 \node (c) at (1.5,2) [colorstyle,label=above:$w$]{};
 \draw [thick] (a1) -- (a2);
 \draw [thick] (a2) -- (a3);
 \draw [thick] (a3) -- (a4);
 \draw [thick] (b1) -- (a2);
 \draw [thick] (b2) -- (a1);
 \draw [thick] (b2) -- (a3);
 \draw [thick] (b3) -- (a2);
 \draw [thick] (b3) -- (a4);
 \draw [thick] (b4) -- (a3);
 \draw [thick] (c) -- (b1);
 \draw [thick] (c) -- (b2);
 \draw [thick] (c) -- (b3);
 \draw [thick] (c) -- (b4);
\end{tikzpicture}
\end{center}

    \caption{The Mycieslkian of $P_4$}
    \label{fig:muP4}
\end{figure}
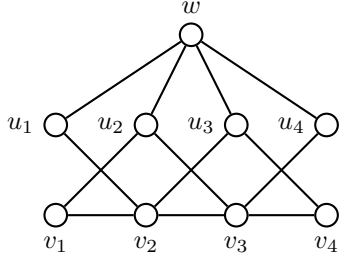

\

\begin{Rem} The new triangles in $\mu(G)$ are of the form $u_iv_jv_k$ where $v_iv_jv_k$ are triangles in $G$. Thus, if G is triangle free, so is $\mu(G)$.
\end{Rem}

\

\begin{Rem} For any graph $G$, the vertices $\{u_1,u_2,\dots,u_n\}$ form an independent set. Hence the graph $\mu(G)$ is not complete. Any graph $G$ therefore satisfies $\gamma_s(\mu(G)) \geq 2$.  
\end{Rem}

There is only one case in which a Mycielskian is not connected; this is when the original graph consists of a single vertex. This result is demonstrated as corollary to the following useful lemma.

\begin{lemma}\label{spanningBound} If H is a spanning subgraph of G, then $\gamma_s(\mu(G)) \leq \gamma_s(\mu(H))$. 

\begin{proof}
  Since $V(H)=V(G)$, $\mu(H)$ is a spanning subgraph of $\mu(G)$. Thus every secure dominating set of $\mu(H)$ is also a secure dominating set of $\mu(G)$ which implies the claim.
\end{proof}

\end{lemma}

\

\begin{cor}
 $\gamma_s(\mu(G)) = 2$ if and only if $G$ is a trivial graph. 
    \begin{proof}
    If $G$ has an edge, it has $n\ge 2$ vertices. Hence $G$ is a spanning subgraph of $K_n$. By Lemma~\ref{spanningBound}, $3=\gamma_s(\mu(K_n) \leq \gamma_s(\mu(G)$. On the other hand, the Mycielskian of a single vertex is a disconnected graph consisting of a path on two vertices, and a single vertex.
    \end{proof}
\end{cor}

\

\section{General Bounds}

A number of conclusions about $\gamma_s(\mu(G))$ can be drawn by considering properties of dominating and secure dominating sets of the underlying graph $G$. This section collects these results.

\subsection{Dominating Vertices}
Recall that we call a vertex \textit{dominating} if it is adjacent to every other vertex. Complete graphs and stars have dominating vertices, as do other graphs. Indeed, a graph with a dominating vertex must contain a star as a subgraph, not necessarily induced. A graph with a dominating vertex has an (ordinary) domination number of $1$.

\begin{thm}\label{secDomMuDominating}
    Let $G$ be a graph of order $n \geq 2$ with a dominating vertex. Then $\gamma_s(\mu(G)) = 3$.
\end{thm}

\begin{proof}
No subsets of cardinality $1$ or $2$ are secure dominating sets of $\mu(G)$. Let $v_1$ be a dominating vertex of $G$, and consider $S = \{v_1, u_1, w\}$ where $u_1$ is the image of $v_1$. The set $S$ is a secure dominating set. All $u_j, j \neq 1$ are defended by $w$ as well as by $v_1$ and all $v_j, j \neq 1$ are defended by $v_1$ as well as by $u_1$. Thus, $\gamma_s(\mu(G)) = 3$.
\end{proof}

\

A vertex adjacent to all but one other vertices could be considered ``nearly-dominating.'' In this case, we can still draw a conclusion about $\gamma_s(\mu(G))$. Recall that $\Delta(G)$ denotes the maximal vertex degree.

\

\begin{thm}
    If $\Delta(G) = n-2$, then $\gamma_s(\mu(G)) \leq 4$.
\end{thm}
\begin{proof}
     Let $v_1$ be the unique vertex of degree $n-2$. Since $G$ is connected, $|G|\geq 4$. Now let $x$ be the unique vertex not adjacent to $v_1$. 
Consider the set $S = \{ v_1, u_1, w, x\}$. We must show that this is a secure dominating set of $\mu(G)$. Clearly $S$ is a dominating set of $\mu(G)$. Let $c \notin S$ be some other vertex. Then we have the following three cases:  
    \begin{enumerate}
    \item $c = v_j, j\neq 1$
    \item $c = u_j, j \neq 1$ and $u_j \neq u_x$, the image of $x$
    \item $c = u_x$ 
    \end{enumerate}
    We now discuss the defenders of $c$ for each case.
    \begin{itemize}
\item If $c = v_j, j \neq 1$ then $v_j$ is securely defended by $v_1$ since $(S \setminus \{v_1\}) \cup \{v_j\} =  \{ v_j, u_1, w, x\}$ is a dominating set of $\mu(G)$. 
 \item $c = u_j, j \neq 1$, then $u_j$ is defended by $v_1$ since $(S \setminus \{v_1\}) \cup \{u_j\} =  \{ u_j, u_1, w, x\}$ is a dominating set of $\mu(G)$.  

 \item If $c = u_x$, then $u_x$ is defended by $w$ since $(S \setminus \{w\}) \cup \{u_x\} =  \{ v_1, u_1, u_x, x\}$ is a dominating set of $\mu(G)$.
 
 Hence S is a secure dominating set of $\mu(G)$. Thus, $\gamma_s(\mu(G)) \leq 4$. 
    \end{itemize}
    
\end{proof}

It is also possible to infer the existence of a dominating vertex in a graph based on secure domination numbers.

\begin{prop}
    
\label{lem:winS}
  If $G$ is not complete of order $n\ge 4$ such that \[ \gamma_s(G) = \gamma_s(\mu(G)) = 3 ,\] then $w \in S$ for every $\gamma_s$-set $S$ of $\mu(G)$.
\end{prop}
\begin{proof}
  Let $S$ be a $\gamma_s$-set of $\mu(G)$ of cardinality 3. If $w \notin S$ then $u_i \in S$ for some $u_i$. Since $n\ge 4$, $S$ cannot contain two vertices from the set $U$, which is the set of images of the vertices in $G$. If so, there are at least two remaining vertices of $U$ to be defended by the remaining one vertex from $S$ which is a contradiction. So $S$ must contain two vertices from $V(G)$ and exactly one from $U$. Since $w$ has to be defended by $u_i$, it means $\mathrm{epn}(u_i,S)= \varnothing$. Each vertex of $V(G)$ not in $S$ is s-defended by one of the two vertices in $S$. Note it cannot be s-defended by $u_i$ as that would leave $w$ undominated. This means that $S \setminus \{u_i\}$ is a secure dominating set of $G$, a contradiction since we are given $\gamma_s(G) = 3$. Hence $w \in S$.
\end{proof}

With the preceding lemma in hand, we can draw a conclusion about dominating vertices in certain graphs.

\begin{prop}
  If $G$ is not complete of order $n \ge 4$ such that \[ \gamma_s(G) = \gamma_s(\mu(G)) = 3,\] then $G$ has a dominating vertex.
\end{prop}
\begin{proof}
  From Proposition ~\ref{lem:winS} we know that any secure dominating set $S$ of $\mu(G)$ contains the vertex $w$ and so $S$ contains two vertices from $V(G) \bigcup U$. Note that $S$ cannot contain two vertices from $V(G)$ as that would mean $S$ is secure in $V(G)$, contradicting that $\gamma_s(G) = 3$. Hence $S$ contains one vertex $v_i$ from $V(G)$, one vertex, say $u_k$, from $U$ and
  $w$. That is ${S = \{v_i,u_k,w\}}$.

  We want to show that $v_i$ is a dominating vertex in G. Since $v_k$ is not adjacent to $u_k$ it means $v_k$ must be adjacent to $v_i$. Also $u_i$ is neither adjacent to $v_i$ nor to $u_k$ so $u_i$ must be s-defended by $w$. This means when $w$ is replaced by $u_i$ then the set $\{v_i, u_i, u_k\}$ must be a dominating set of $\mu(G)$. So all other $u_r$ where $r \ne i, k$ must be adjacent to $v_i$. By the construction of $\mu(G)$, this means $v_r$ where $r \ne i, k$ must be adjacent to $v_i$. Since $v_k$ is already shown to be adjacent to $v_i$, all this shows that $v_i$ is a dominating vertex of $G$.
  
  We now want to show that the set $S = \{v_i, u_k, w\}$ is a secure dominating set of $\mu(G)$ if $v_k$ (the pre-image of $u_k$) is also a dominating vertex in $G$. If $v_k$ is not dominating then the vertex $v_r$ not adjacent to $v_k$ will be left undominated when $v_i$ is replaced by $v_k$. Hence the result. Thus a minimum secure dominating set of $\mu(G)$ is of the type $\{v_i, u_k, w\}$ where $v_i$ and $v_k$ (the pre-image of $u_k$) is a dominating vertex in $G$.
\end{proof}
\subsection{$S$-isolates}

In this section, we give a sufficient condition under which a secure dominating set $S \subset G$ implies that $S \cup \{w\}$ is a secure dominating set of $\mu(G)$. We do so using the language of ``$S$-isolates."

\begin{definition}
    Let $S$ be a subset of vertices of $G$, and let $v$ be a vertex in $S$. We call the vertex $v$ an \textit{$S$-isolate} if $v$ is an isolate in the subgraph induced by $S$. That is, $N(v) \cap S = \emptyset$. 
\end{definition}

\

$S$-isolates in a graph $G$ can be described as external private neighbors of the vertex $w \in \mu(G)$. The specific claim is the following one.

\begin{lemma}
    Let $S$ be a dominating set of $G$, considered also as a subset of vertices of the graph $\mu(G)$. The external private neighbors of $w \in \mu(G)$ with respect to $S \cup \{w\}$ is the set of images of $S$-isolates (in $G$).
    
    That is,
    \[ \epn(w,S\cup\{w\}) = \{ u_i, \text{ where }~ v_i\in S \mid N(v_i) \cap S = \emptyset \} \]
\end{lemma}

\begin{proof}
Let $S \subset G$ a dominating set of $G$.

Let $c \in \epn(w, S \cup \{w\})$. Since $N(w) = U$, $c = u_i$ for some $i.$ Since $S$ is a dominating set of $G$, either $v_i \in S$ or $N(v_i) \cap S \neq  \emptyset.$ 

Since $c =u_i \in \epn(w, S \cup \{w\})$, it follows that $N(u_i) \cap S =\emptyset$. Hence, by the construction of $\mu(G)$, $ N(v_i) \cap S = \emptyset$. Since S is a dominating set, this means that $v_i \in S$. Also, $N(v_i)\cap S = \emptyset$ means that $v_i$ is an isolate of $S$. 

Now let $u_i \in \{u_k \in U \mid N(v_k) \cap S = \emptyset\}$. That is, $u_i$ is the image of an S-isolate $v_i$. Since $N(u_i) \cap S = N(v_i) \cap S = \emptyset$ this means that $u_i$ is not agacent to any vertex of $S$. We know that $u_i$ is a neighbor of $w$, and thus $u_i$ is in $\epn(w, S \cup \{w\})$. 
\end{proof}

\begin{figure}[H]
    \centering
    \begin{center}
      \begin{tikzpicture}[scale=0.9, colorstyle/.style={circle,draw=black!100,thick,inner sep=2pt, minimum size=3mm}]
 \node (a1) at (0,0) [colorstyle,label=below:$v_1$]{};
 \node (a2) at (1,0) [colorstyle,fill=black,label=below:$v_2$]{};
 \node (a3) at (2,0) [colorstyle,label=below:$v_3$]{};
 \node (a4) at (3,0) [colorstyle,fill=black,label=below:$v_4$]{};
 \node (a5) at (4,0) [colorstyle,label=below:$v_5$]{};
 \node (b1) at (0,1) [colorstyle,label=left:$u_1$]{};
 \node (b2) at (1,1) [colorstyle,fill=red,label=left:$u_2$]{};
 \node (b3) at (2,1) [colorstyle,label=left:$u_3$]{};
 \node (b4) at (3,1) [colorstyle,fill=red,label=left:$u_4$]{};
 \node (b5) at (4,1) [colorstyle,label=right:$u_5$]{};
 \node (c) at (2,2) [colorstyle,label=above:$w$]{};
 \draw [thick] (a1) -- (a2);
 \draw [thick] (a2) -- (a3);
 \draw [thick] (a3) -- (a4);
 \draw [thick] (a4) -- (a5);
 \draw [thick] (b1) -- (a2);
 \draw [thick] (b2) -- (a1);
 \draw [thick] (b2) -- (a3);
 \draw [thick] (b3) -- (a2);
 \draw [thick] (b3) -- (a4);
 \draw [thick] (b4) -- (a3);
 \draw [thick] (b4) -- (a5);
 \draw [thick] (b5) -- (a4);
 \draw [thick] (c) -- (b1);
 \draw [thick] (c) -- (b2);
 \draw [thick] (c) -- (b3);
 \draw [thick] (c) -- (b4);
 \draw [thick] (c) -- (b5);
\end{tikzpicture}
\end{center}

    \caption {A dominating set in $P_5$ and private neighbors of $w$}
    \label{fig:enter-label}
\end{figure}
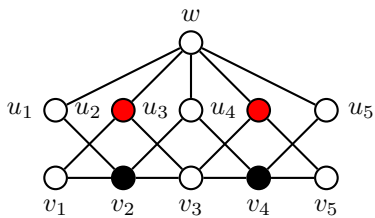

Proposition~\ref{Swepn}, to be considered shortly, indicates the significant of isolates in dominating sets for the question of secure domination Mycielskians. In the investigation of that relation, the following result, which makes no use of Mycielskians, was discovered, and could be of independent interest. 

\begin{prop} If a minimal secure dominating set $S$ for a path $P_n$ with $n \geq 6$ has no isolates, then $S$ contains no set of three adjacent vertices.
\end{prop}
\begin{proof}
  Consider the contrapositive. Enumerate the vertices $v_1, v_2, \cdots, v_n$. It is
  immediate that a minimal secure dominating set will not contain \{$v_1, v_2, v_3\}$.
  Now consider an arbitrary $\{v_j, v_{j+1}, v_{j+2}\}$. Since $n\geq 6$ there is at least one other guard. Suppose the guard is “to the right,” i.e. at $v_{j+5}$. If this guard
  is not lonely, i.e. if $v_{j+6}$ is also guarded, then we can remove $v_{j+1}$ and still
  have a secure dominating set (by moving the guard at $v_{j+5}$ to $v_{j+4}$), which
  contradicts minimality.
\end{proof}

\begin{prop}\label{Swepn}
    Let $G$ be a connected graph of order at least 2. If $S$ is a secure dominating set of $G$, then 
    \[ S_p = S \cup \epn(w,S\cup\{w\}) \cup \{w\}\] is a secure dominating set of $\mu(G)$.
\end{prop}

\begin{proof}
    Since $S$ is a dominating set of $G$ and $N[w] = \{w\} \cup U$, it follows that $S \cup \{w\} \subset S_p$ is a dominating set of $\mu(G).$ To show that $S_p$ is a secure dominating set of $\mu(G)$, let $c \notin S_p$ be some other vertex. Since $w \in S_p$, it must be that $c \in V(G) -S$ or $c \in U$. 

    If $c \in V(G)-S$, $c$ has the same $s$-defender in $\mu(G)$ with respect to $S_p$ as in $G$ with respect to $S$. That is, since $S$ is a secure dominating set of $G$, there is an $s \in N(c) \cap S$ such that $N[S - \{s \} \cup \{c\}]= V(G)$. Thus, in $\mu(G)$, 
    
    \begin{align*}
        S_p - \{s\} \cup \{c\} & = \{w\} \cup S - \{s\} \cup \{c\} \cup \epn(w, S \cup \{w\}) \\
        & \supset N_{\mu(G)}[w] \cup N_{\mu(G)}[S - \{s\} \cup \{c\} ] \\
        & = \{w\} \cup U \cup V(G) 
    \end{align*}

If $c \in U$, then note $w$ securely defends $c$ since $\{c \} \cup \epn(w, S \cup \{w\}) \cup S$ is a dominating set of $\mu(G)$.
\end{proof}

\begin{cor} \label{SandIsos}
Let $S \subset G$ be a secure dominating set. Let $I_S$ be the set of $S$-isolates in $S$. Then 
    \[ \gamma_s(\mu(G)) \leq |S| + |I_S| + 1  \]
    
\end{cor}

\begin{proof}
Note $epn(w, S \cup \{w\}) = \{u_i \mid N(v_i) \cap S = \emptyset \}$. 
We use Proposition \ref{Swepn} to note that 
\begin{align*}
S \cup \epn(w,S\cup \{w\}) \cup \{w\} & = S \cup \{u_i \mid  N(v_i) \cap S = \emptyset\} \cup \{w\}
\end{align*}
is a secure dominating set of $\mu(G)$. Hence the secure domination number of $\mu(G)$ is at most $|S|+|I_S|+1$. \end{proof}

We are therefore interested in isolates because their absence indicates that the Mycielskian of a graph has a simple secure dominating set: the union of the given secure dominating set of $G$ with the central vertex $w$ of the Mycielski construction.

\nocite{*}

\subsection{Ordinary Domination Bound}
By ``doubling'' an ordinary dominating set, it is possile to obtain an upper bound for the secure domination number of the Mycielskian. This result appears as Theorem 11 in \cite{NaRa}.

\begin{prop}\label{prop:domBound} \cite{NaRa}
  $\gamma_s(\mu(G)) \le 2\gamma(G) + 1$
\end{prop}

\begin{proof}
  Let $S$ be a dominating set of $G$. Let $S'$ be the set of images of those vertices. Then ${S \bigcup S' \bigcup \{w\}}$ is a secure dominating set of $\mu(G)$.
\end{proof}

The upper bound is attained for the path on six vertices.

\section{Special Families}
We compute here the secure domination numbers of Mycielskians of familiar graph families.

\subsection{Paths and Cycles}

The secure domination number of the \textit{generalized } Mycielskian of path graphs has been published for path graphs of small order \cite{NaRa}. 
 We give the general result for the (standard) Mycielskian of paths of arbitrary order.

Recall Theorem~\ref{secDomPath}; the secure domination number of a path on $n$ vertices is $\lceil \frac{3n}{7} \rceil$. A minimal secure dominating set for such a path comes from taking three non-adjacent vertices out of every seven. The isolation of these vertices indicates that the bound coming from Corollary \ref{SandIsos} is not particularly good. In fact, we can do much better by pairing vertices. A slightly larger secure dominating set of $P_n$ yields a much smaller secure dominating set of $\mu(P_n)$.

The path $P_{7}$ illustrates this phenomenon. By Theorem~\ref{secDomPath}, ${\gamma_s(P_{7}) = \lceil {\frac{3 \cdot 7}{7}} \rceil = 3}$. Figure \ref{fig:P7inefficient} shows the (unique) minimal secure dominating set of $P_{7}$. 

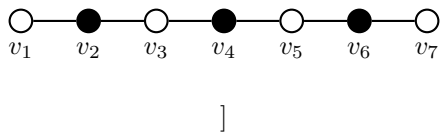
\begin{figure}[h]
    \centering
\begin{center}
\begin{tikzpicture}[scale=0.9, colorstyle/.style={circle,draw=black!100,thick,inner sep=2pt, minimum size=3mm}]
 \node (a1) at (0,0) [colorstyle,label=below:$v_1$]{};
 \node (a2) at (1,0) [colorstyle,fill=black!100,label=below:$v_2$]{};
 \node (a3) at (2,0) [colorstyle,label=below:$v_3$]{};
 \node (a4) at (3,0) [colorstyle,fill=black,label=below:$v_4$]{};
 \node (a5) at (4,0) [colorstyle,label=below:$v_5$]{};
 \node (a6) at (5,0) [colorstyle,fill=black!100,label=below:$v_6$]{};
 \node (a7) at (6,0) [colorstyle,label=below:$v_7$]{};
 \draw [thick] (a1) -- (a2);
 \draw [thick] (a2) -- (a3);
 \draw [thick] (a3) -- (a4);
 \draw [thick] (a4) -- (a5);
 \draw [thick] (a5) -- (a6);
 \draw [thick] (a6) -- (a7);
\end{tikzpicture}
\end{center}

]
    \caption{A good dominating set for $P_7$ that is inefficient for $\mu(P_7)$.}
    \label{fig:P7inefficient}
\end{figure}

Applying Corollary~\ref{SandIsos} yields a bound of the form $\gamma_s(\mu(P_7)) \le 3 + 3 + 1 = 7$. This bound, however, is too large.

Consider the secure dominating set of Figure~\ref{fig:pairedGuards}, which is not minimal. The secure dominating set of Figure~\ref{fig:pairedGuards} consists only of paired guards. Therefore Corollary~\ref{SandIsos} gives an upper bound of $4+0+1=5$ for the secure domination number of $\mu(P_7)$, and this is in fact the value of the parameter.

\begin{figure}
  \centering
\begin{center}
\begin{tikzpicture}[scale=0.9, colorstyle/.style={circle,draw=black!100,thick,inner sep=2pt, minimum size=3mm}]
 \node (a1) at (0,0) [colorstyle,label=below:$v_1$]{};
 \node (a2) at (1,0) [colorstyle,fill=black!100,label=below:$v_2$]{};
 \node (a3) at (2,0) [colorstyle,fill=black!100,label=below:$v_3$]{};
 \node (a4) at (3,0) [colorstyle,label=below:$v_4$]{};
 \node (a5) at (4,0) [colorstyle,fill=black!100,label=below:$v_5$]{};
 \node (a6) at (5,0) [colorstyle,fill=black!100,label=below:$v_6$]{};
 \node (a7) at (6,0) [colorstyle,label=below:$v_7$]{};
 \node (b1) at (0,1) [colorstyle,label=left:$u_1$]{};
 \node (b2) at (1,1) [colorstyle,label=left:$u_2$]{};
 \node (b3) at (2,1) [colorstyle,label=left:$u_3$]{};
 \node (b4) at (3,1) [colorstyle,label=left:$u_4$]{};
 \node (b5) at (4,1) [colorstyle,label=right:$u_5$]{};
 \node (b6) at (5,1) [colorstyle,label=right:$u_6$]{};
 \node (b7) at (6,1) [colorstyle,label=right:$u_7$]{};
 \node (c) at (3,2) [colorstyle,fill=red,label=above:$w$]{};
 \draw [thick] (a1) -- (a2);
 \draw [thick] (a2) -- (a3);
 \draw [thick] (a3) -- (a4);
 \draw [thick] (a4) -- (a5);
 \draw [thick] (a5) -- (a6);
 \draw [thick] (a6) -- (a7);
 \draw [thick] (b1) -- (a2);
 \draw [thick] (b2) -- (a1);
 \draw [thick] (b2) -- (a3);
 \draw [thick] (b3) -- (a2);
 \draw [thick] (b3) -- (a4);
 \draw [thick] (b4) -- (a3);
 \draw [thick] (b4) -- (a5);
 \draw [thick] (b5) -- (a6);
 \draw [thick] (b5) -- (a4);
 \draw [thick] (b6) -- (a7);
 \draw [thick] (b6) -- (a5);
 \draw [thick] (b7) -- (a6);
 \draw [thick] (c) -- (b1);
 \draw [thick] (c) -- (b2);
 \draw [thick] (c) -- (b3);
 \draw [thick] (c) -- (b4);
 \draw [thick] (c) -- (b5);
 \draw [thick] (c) -- (b6);
 \draw [thick] (c) -- (b7);
\end{tikzpicture}
\end{center}
  
  \caption{Pairing guards in $P_7$ yields an efficient secure dominating set for $\mu(P_7)$.}
  \label{fig:pairedGuards}
\end{figure}
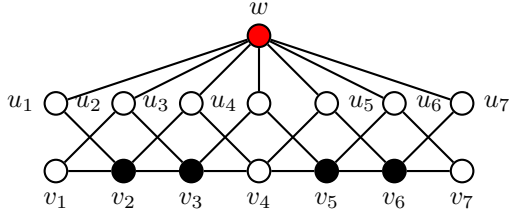

The following theorem generalizes the reasoning just given; to produce a secure dominating set of $\mu(P_n)$, choose pairs of guards in the original path, and add the cone vertex $w$. The cycle $C_n$ contains $P_n$ as a spanning subgraph, therefore Lemma~\ref{spanningBound} makes it natural to compute the secure domination numbers of the two graph families at the same time.

\begin{thm}\label{secDomMuPath}
  $$\gamma_s(\mu(P_n)) =\gamma_s(\mu(C_n))=
  \begin{cases}
    2k + 1 & n= 4k \\
    2k + 2 & n= 4k + 1 \\
    2k + 3 & n= 4k + 2 \mathrm{~or~} 4k+3\\
  \end{cases}
  $$
\end{thm}

\begin{proof}
  \phantom{1}

  We first bound $\gamma_s(C_n)$ below.
  
  \textbf{Case 1: $n=4k$}
  
  Let $S$ be a set of vertices of $\mu(C_n)$ of size $2k$. We prove by contradiction that $S$ is not secure dominating.

  \textbf{Case 1a:} Suppose that $w \notin S$. An attack at $w$ must be covered, say by a guard at $u_i$. This leaves only $2k-1$ guards among the $v$ and $u$ vertices. That many guards are adjacent to $2(2k-1) = 4k-2$ of the $u$ vertices, which means that at least one of the $u$ vertices was not defended by $S$.

  \textbf{Case 1b:} Suppose that $w \in S$. Let $\hat{S}$ denote the set $S \backslash \{w\}$. Observe that $\mathrm{epn}(w,\hat{S}) \ge 2$, since $4k-2$ of the $u$ vertices are defended by $\hat{S}$. Thus an attack at one of those private neighbors cannot be defended by the guard at $w$ without abandoning another private neighbor.

  \textbf{Case 2: $n=4k + 1$}
  Let $S$ be a set of vertices of $\mu(C_n)$ of size $2k + 1$. We prove by contradiction that $S$ is not secure dominating.

  \textbf{Case 2a:} Suppose that $w \notin S$. At least one $u$ vertex must be in $S$, so that an attack at $w$ can be covered. At most one $u$ vertex can be in $S$, since otherwise there are not enough guards among the $v$ vertices to cover all $4k+1$ $u$ vertices. Then there are $2k$ guards among the $v$ vertices. For these to cover $4k$ $u$ vertices, they must come in pairs separated by $2$. But then an attack at a $u$ vertex above such a pair disrupts the configuration, since it leaves a gap of $3$ in the path of $v$ vertices.

  \textbf{Case 2b:} Suppose $w \in S$. Let $\hat{S}$ be as above. Then $\mathrm{epn}(w, \hat{S}) \ge 1$, with equality holding only if the guards among $v$ are in pairs separated by two empty vertices. An attack at a $u$ vertex above such a pair cannot be defended either by the guard at $w$ or the guards at the $v$ vertices.

  \textbf{Case 3: $n=4k+2$}
  Let $S$ be a set of vertices of $\mu(C_n)$ of size $2k + 2$. We prove by contradiction that $S$ is not secure dominating.
  
  \textbf{Case 3a:} Suppose that $w \notin S$.
  To cover an attack at $w$ and also cover an attack an any $u$ vertex means that $S$ has $2k+1$ guards all among the $v$ vertices (i.e. all but one guard). These must be distributed in pairs of two, separated by gaps of two, with an additional isolated guard. An attack above one of the pairs cannot be securely guarded, as before.
  
    \textbf{Case 3b:} Suppose $w \in S$.
    By arguments like those preceding, the vertices of $S$ among the $v$ must come in groups of two, separated by two, together with a singleton. Consider an attack at a $u$ vertex above a pair. If it is covered from a $v$ guard, then there is a large gap downstairs, with three adjacent undefended vertices. If it is covered from the $w$ guard, then the $u$ vertex above the $v$ singleton is undefended.
    
  \textbf{Case 4: $n=4k+3$}
  Let $S$ be a set of vertices of $\mu(C_n)$ of size $2k + 2$. We prove by contradiction that $S$ is not secure dominating.

  \textbf{Case 4a:} Suppose that $w \notin S$. To cover an attack at $w$ as well as at any $u$ means that there must be $2k+1$ guards among the $v$ vertices (and one among the $u$). Then, as before, an attack above a pair of adjacent $v$-guards (i.e. at one of the corresponding $u$ vertices) cannot be covered.

  \textbf{Case 4b:} Suppose $w \in S$. There are $2k+1$ guards at $v$ vertices guarding $4k+2$ $u$ vertices. There is either an isolated guard among the $v$, or redundancy becuase they come in a block of at least three. As before, an attack above a pair cannot be covered.

We have now shown that the secure domination number of the Mycielskian of a cycle is at least as large as indicated. We give an upper bound now, by showing that the indicated value is in fact sufficient for paths (which are spanning subgraphs of the corresponding cycles).
  
\textbf{Sufficiency:} In the case that $n=4k$, place guards in pairs at the center of every block of four vertices in $P_n$. In the case that $n=4k+1$, place guards as for $n=4k$, but add three guards in the center of some block of five vertices. In the cases $n=4k+2$ and $n=4k+3$ place pairs but with a smaller gap to the next adjacent block. Then these sets, which have no isolates, together with $w$, form a secure dominating set, by Proposition~\ref{Swepn}.
    \end{proof}
    
\ 

\subsection{Complete Bipartite Graphs}

The secure domination number of the Mycielskian of a complete graph is $3$ (Theorem~\ref{secDomMuDominating}). We now give the result for complete bipartite graphs.

Recall the secure domination numbers of complete bipartite graphs.

\begin{thm}\cite{ck}
Let $m \le n$ be natural numbers.
$$\gamma_s(K_{m,n})=
  \begin{cases}
    n &m=1 \\
    2 &m=2 \\
    3 &m=3 \\
    4 &m\ge 4 \\
  \end{cases}
  $$
  
\end{thm}

We have the following for Mycielskians of complete bipartite graphs.

\begin{thm}
Let $m \le n$ be natural numbers.
$$\gamma_s(\mu(K_{m,n}))=
  \begin{cases}
    3 &m=1 \\
    4 &m=2\mathrm{~or~}3 \\
    5 &m\ge 4 \\
  \end{cases}
  $$
\end{thm}
\begin{proof}
  The case $m=1$ is an instance of Theorem~\ref{secDomMuDominating}, since the graph has a dominating vertex. In the other cases, we omit necessity, which is a straightforward but tedious check, and consider only sufficiency.

  Label the vertices so that $\{v_1,v_2,\dots,v_m\}$ constitute one bipartite part (the smaller), and $\{v_{m+1},v_{m+2},\dots,v_{m+n}\}$ form the second part. Let $V_m$ denote the former set, $V_n$ the latter, and let $U_m$ and $U_n$ be the respective sets of images.
  
  In the case $m=2$ and $m=3$, the set $S=\{v_1,v_2,v_4,w\}$ is a secure dominating set, which is seen as follows. The vertex $v_3$ (if $m=3$) is defended by the guard at $v_4$, any vertex of $V_n$ is defended by the guard at $v_2$, and any vertex of  $U_m$ or $U_n$ is defended by the guard at $w$.

  For $m \ge 4$ we can reason similarly. Consider the set $S = \{v_1,v_2,v_{m+1},v_{m+2},w\}$. Any vertex of $V_m$ is defended by the guard at $v_{m+2}$, any vertex of $V_n$ is defended by the guard at $v_2$, and any vertex of  $U_m$ or $U_n$ is defended by the guard at $w$.

\end{proof}

\section{Constructions for given parameters}

It is possible to construct graphs whose secure domination number bears a specified relation to the secure domination number of the Mycielskian of the graph. We collect relevant constructions here.

\subsection{Prescribed Gaps}

Given a natural number $k$, we can ask for a graph $G$ such that $|\gamma_s(\mu(G)) - \gamma_s(G)| = k$. Two simple families of graphs solve this problem completely: paths and stars. The next subsection considers the more challenging problem of prescribing values of the parameters separately, and those results subsume these.

First we consider the case that the secure domination number of the Mycielskian is greater than the secure domination number of the original graph.

\begin{prop}
 For any $k \in \mathbb{N}$, there is a graph $G$ such that ${\gamma_s(\mu(G)) - \gamma_s(G) = k}$. 
\end{prop}
       
\begin{proof}
Let $k\ge 2$ be a natural number. If $k$ is even, then $P_{14k-26}$ is a path graph (the smallest, in fact) such that $ \gamma_s(\mu(G)) - \gamma_s(G)= k$. If $k$ is odd, then $P_{14k-28}$ is a path graph (the smallest, in fact) such that ${\gamma_s(\mu(G)) - \gamma_s(G) = k}$.

  For the case $k=1$, consider $P_3$.
  
\end{proof}

\begin{prop}
  For any $k \in \mathbb{Z}_{\ge 0}$, there is a graph $G$ such that ${\gamma_s(G) - \gamma_s(\mu(G)) = k}$
\end{prop}

\begin{proof}
  A star $G=K_{1,k+3}$ satisfies $\gamma_s(G) - \gamma_s(\mu(G)) = k$.
\end{proof}

\subsection{Prescribed Values}

We can ask to prescribe the parameter values individually and not just the gap between them. The following theorem largely solves this problem, excluding extreme cases. Note the relevance of Proposition~\ref{prop:domBound}. The secure domination number of the Mycielskian cannot be significantly greater than the secure domination number of the original graph; more precisely, ${\gamma_s(\mu(G)) \le 2\gamma(G) + 1 \le 2\gamma_s(G) + 1}$. We are able at present to realize in general every possibility except ${\gamma_s(\mu(G)) = 2\gamma_s(G)}$ and ${\gamma_s(\mu(G)) = 2\gamma_s(G)+1}$.

\begin{thm}\label{thm:prescVal}
  Given a natural number $a\ge 2$ and a natural number $b$ satisfying ${3 \le b \le 2a - 1}$, there is a graph $G$ such that $\gamma_s(G) = a$ and $\gamma_s(\mu(G))=b$.
\end{thm}
\begin{proof}
  We introduce two constructions according to the relative sizes of the parameters $a$ and $b$.
  
  \textit{Case 1: $b > a$}
  
  Construct a graph $G$ in the following manner. Consider a complete graph on $b-a$ vertices, with vertex set $\{v_1,v_2,\dots,v_{b-a}\}$. To each of these vertices, attach a subdivided leaf. In other words, add vertices $\{l_i^0,l_i^1\}$, an edge $(v_il_i^0)$, and an edge $(l_i^0l_i^1)$. Add an additional ${2a-b-1}$ subdivided leaves (with vertices $\{m_j^0,m_j^1\}$ to the vertex $v_1$. This is the graph $G$.
  First, we show that $\gamma_s(G) = a$. The set $$\{l_1^0,l_2^0, \dots, l_{b-a},m_1^0,m_2^0,\dots,m_{2a-b-1}^0,v_1\}$$ is a secure dominating set. The guards among the subdivided leaves defend the degree-one vertices, and the guard $v_1$ defends the vertices of the complete subgraph induced by $\{v_i\}$. It is clear too that $a$ guards are necessary, since each leaf needs a guard, and one additional guard is needed to defend the high-degree vertices.
  Second, we show that $\gamma_s(\mu(G)) = b$. The set $$\{l_1^0,l_2^0, \dots, l_{b-a},m_1^0,m_2^0,\dots,m_{2a-b-1}^0,v_1,v_2,\dots,v_{b-a},w\}$$ does securely defend $\mu(G)$, where $w$ is the cone vertex of the Mycielskian as before. The vertices within $G$ induce a connected graph, so by Proposition~\ref{Swepn} we only need to add $w$ to obtain a secure dominating set of $\mu(G)$. That $b$ guards are necessary can be seen in the following manner. Since the image of each leaf (i.e. the images of each $l_i^1$ and $m_j^1$) is defended only by $w$ and the lower vertex adjacent to the leaf, and the latter is also the sole defender of the leaf itself, there are necessarily two guards for each subdivided leaf in the graph, with the possibility of duplication in the case that the vertices of the complete graph, where the subdivided leaves attach, are included. This gives $2(b-a) + (2a-b-1) = b-1$ guards needed among the lower level, and one additional guard is needed as well.

\textit{Case: $b \le a$}

In the case that the secure domination number of the Mycielskian is not greater than that of the original graph, path with many leaves added to it serves the purpose.

Take a path on $b-2$ vertices $\{v_1,\dots,v_{b-2}\}$. To each of these vertices add a leaf $l_i$. In other words, include a vertex $l_i$ and an edge $(v_i,l_i)$. Add an additional $a-b+2$ $m_j$ adjacent to $v_1$. This is the graph $G$.

First, we show that $\gamma_s(G) = a$. Each leaf requires a defender, and there are ${(a-b+2) + (b-2) = a}$ such leaves.

Second, we show that $\gamma_s(\mu(G)) = b$. Consider the set consisting of all $v_i$, the image of $v_1$, and $w$. This set, with cardinality ${(b-2) + 1 + 1 = b}$, is evidently a secure dominating set, since anything at the lower level is securely defended by the adjacent $v_i$ guard, and anything above (i.e. any image vertex) is securely defended by $w$. No fewer suffice, since each pair  $(v_il_i)$ with $i \ge 2$ requires a defender, the vertices around $v_1$ require two defenders, and one additional defender $w$ is also necessary.
\end{proof}

\vskip.4in

\section{Future Work}

We have established a number of results relating the Mycielski construction and secure domination. Various avenues for continued investigation remain open.

\begin{enumerate}
\item Let $G$ be a graph and $a, b$ and $c$ be three natural numbers where $a \leq b$ and $a+1\leq c \leq 2a+1$, such that $\gamma(G)=a, \gamma_s(G) = b$ and $ \gamma_s(\mu(G)) =c$. If $a \geq 2$ and $c=a+1$, must $b=a$?
\item Can we use $\Delta(G)$ to bound $\gamma_s(\mu(G))$? More precisly, does $\Delta(G) = n-k$ imply that $\gamma_s(\mu(G)) \leq k+2$?
\item Can the open cases for the problem of Theorem~\ref{thm:prescVal} be solved? In other words, are there graphs $G$ for which $\gamma_s(\mu(G))= 2\gamma_s(G) = 2k$ for, for arbitrary $k$? Are there graphs $G$ for which $\gamma_s(\mu(G))= 2\gamma_s(G)+1 = 2k+1$ for, for arbitrary $k$?
\item Characterize graphs for which $\gamma_s(\mu(G)) = \gamma_s(G) $. Note that a family of such graphs is given in Theorem~\ref{thm:prescVal}.

\end{enumerate}

\printbibliography

\bigskip

=
\bigskip

\noindent
\textbf{Kiran R. Bhutani}\\
Department of Mathematics and Statistics, The Catholic University of America, 
620 Michigan Avenue NE, Washington, DC 20064, USA.\\
E-mail: bhutani@cua.edu

\vspace{1em}

\noindent
\textbf{Anthony Christiana}\\
Department of Mathematics, 
The George Washington University, 
800 22nd Street NW, Washington, DC 20052, USA.\\
E-mail: ajchristiana@gwu.edu

\vspace{1em}

\noindent
\textbf{Peter Ulrickson}\\
Department of Mathematics and Statistics, The Catholic University of America, 
620 Michigan Avenue NE, Washington, DC 20064, USA.\\
E-mail: ulrickson@cua.edu
\end{document}